\documentclass[12pt]{amsart}
\usepackage{amssymb, amscd, amsmath, amsthm, epsfig, latexsym, enumerate}

\renewcommand{\geq}{\geqslant}
\renewcommand{\leq}{\leqslant}

\newtheorem{theorem}{Theorem}
\newtheorem{lemma}[theorem]{Lemma}
\newtheorem{cor}[theorem]{Corollary}

\newtheorem*{thm}{Theorem}
\newtheorem*{lem}{Lemma}
\newtheorem*{cor*}{Corollary}

\begin{document}
\title{Nilpotent groups with balanced presentations}
 
\author{J.A.Hillman}
\address{School of Mathematics and Statistics\\
     University of Sydney, NSW 2006\\
      Australia }

\email{jonathanhillman47@gmail.com}

\begin{abstract}
We show that if a nilpotent group $G$ has a balanced presentation 
and Hirsch length $h(G)>3$ then $\beta_1(G;\mathbb{Q})=2$.
There is one such group which is torsion-free and of Hirsch length $h=4$,
and none with $h=5$.
We also construct a torsion-free nilpotent group $G$ with $h(G)=6$ and
$\beta_2(G;F)=\beta_1(G;F)$ for all fields $F$.
\end{abstract}

\keywords{balanced, Betti number, Hirsch length, nilpotent}

\subjclass{20F18, 20J05}

\maketitle


A finite presentation for a group $G$ is {\it balanced\/}
if it has an equal number of generators and relations.
This notion has been studied in connection with finite groups, 
and is central to the Andrews-Curtis conjecture on presentations 
of the trivial group, 
but here we shall consider finitely generated infinite nilpotent groups.
It is well-known that finitely generated torsion-free nilpotent groups of Hirsch length $\leq3$ are either free abelian or are central extensions 
of $\mathbb{Z}^2$ by $\mathbb{Z}$, 
and have balanced presentations.
Work by Lubotzky and by Freedman,  Hain and Teichner relating to 
the Golod-Shafarevich Theorem shows that if a finitely generated 
nilpotent group $G$ has a balanced presentation then 
$\beta_1(G;\mathbb{Q})\leq3$ \cite{FHT,Lu83}.
Freedman,  Hain and Teichner conjecture that if $G$ is such a group and
$\beta_1(G)=3$ then $G$ is virtually $\mathbb{Z}^3$.
We shall confirm this conjecture, and show that there is
one torsion-free nilpotent group with Hirsch length $h(G)=4$ and 
$\beta_2(G;\mathbb{Q})=\beta_1(G;\mathbb{Q})=2$, 
none with $h(G)=5$ and several with $h(G)=6$.

Theorem \ref{h=4} and Corollary \ref{h=5} below were originally part of \cite{Hi19}, 
in which it is observed that if the fundamental groups $\pi_X$ and $\pi_Y$
of the complementary regions of a closed hypersurface in $S^4$ 
are both solvable (or even elementary amenable) then they satisfy the condition
$\beta_2(G;F)\leq\beta_1(G;F)$, for $G=\pi_X$ or $\pi_Y$ and all fields $F$.
Abelian groups satisfying this condition are easily listed.
Nilpotent groups are particularly amenable to homological arguments
(cf. \cite{Dw75}), and so it is natural to consider them next.
(In the topological setting other issues become intractable
if there is non-trivial torsion.
Our arguments below allow us to work with torsion-free groups,
for the most part.)

In \S1 we set out our notation and prove five lemmas.
(These are mostly well-known, but are included for convenience, 
as their proofs are short.)
Our main result is in \S2, where
Theorem \ref{beta=3} proves that if $\beta_2(G;k)=\beta_1(G;k)=3$ 
for all prime fields $k$ then $G\cong\mathbb{Z}^3$.
In particular, if $\beta_1(G;\mathbb{Q})=3$
and $h(G)>3$ then $\beta_2(G;\mathbb{Q})>\beta_1(G;\mathbb{Q})$.
Thereafter we assume that $\beta_1(G;\mathbb{Q})=2$. 
In \S3 we show that there is just one finitely generated,
torsion-free nilpotent group $G$
with $\beta_2(G;\mathbb{Q})=\beta_1(G;\mathbb{Q})$ and $h(G)=4$
(Theorem \ref{h=4}), 
and none with $h(G)=5$ (Corollary \ref{h=5}).
In \S4 we show that there are no other finitely generated, 
metabelian nilpotent groups $G$ with $h(G)>4$ and
$\beta_2(G;\mathbb{Q})=\beta_1(G;\mathbb{Q})$ (Theorem \ref{metabelian2}).
In \S5 we construct a torsion-free nilpotent group $G$
with Hirsch length 6 and 
$\beta_2(G;F)=\beta_1(G;F)=2$ for all fields $F$.

In the final section we consider briefly the related work in Lie theory
\cite{CJP,Ni83,No54,Rag}
which prompted us to look beyond the cases with $h\leq4$.
We outline why there are other nilpotent groups
with Hirsch length 6 and $\beta_2(G;\mathbb{Q})=\beta_1(G;\mathbb{Q})=2$,
but we do not attempt to construct further explicit examples.

\section{preliminaries}

If $G$ is a group $\zeta{G}$, $G'$, $G^\tau$ and $\gamma_nG$ 
shall denote the centre, the commutator subgroup,
the preimage in $G$ of the torsion subgroup of $G/G'$ and the $n$th stage 
of the lower central series, respectively. 
Our convention for commutators is $[x,y]=xyx^{-1}y^{-1}$.
The nilpotency class of  $G$ is $n$ if 
$\gamma_nG\not=1$ and $\gamma_{n+1}G=1$.
If $G$ is a finitely generated nilpotent group then it has a finite composition series 
with cyclic factors, and the Hirsch length $h(G)$ is the number of infinite factors in such a series.
(If $A$ is an abelian group then $h(A)=
\mathrm{dim}_\mathbb{Q}\mathbb{Q}\otimes{A}$ is the rank of $A$.)
If $G$ is torsion-free and not cyclic then $G$ has nilpotency class 
$\leq{h(G)}+1-h(G/G')$, which is strictly less than $h(G)$.

As most of our results involve comparing $\beta_2(G;\mathbb{Q})$ 
with $\beta_1(G;\mathbb{Q})$, for various groups $G$,
we shall simplify the notation by dropping the coefficients, 
except where they seem needed for clarity.
Thus $H_i(G)$, $H^i(G)$ and $\beta_i(G)$ shall denote homology, 
cohomology and Betti numbers with coefficients $\mathbb{Q}$, respectively.
We shall often simplify the notation further by setting $\beta=\beta_1(G)$.
We use also without further comment the facts that if $G$ is a finitely generated nilpotent group with $h(G)=h$ then $G$  
is an orientable $PD_h$-group (Poincar\'e duality group of formal dimension $h$) 
over $\mathbb{Q}$, 
and that $\chi(G)=0$. 
See \cite[Theorem 9.10]{Bi}.

A group $G$ is {\it homologically balanced\/} if
$G$ is finitely generated and $\beta_2(G;R)\leq\beta_1(G;R)$ for all fields $R$.
If $G$ has a balanced presentation then it is homologically balanced.
However, there are groups $G$ such that 
$H_1(G;\mathbb{Z})=H_2(G;\mathbb{Z})=0$, 
and which have no presentation of deficiency 0 \cite{HW85}.
Such groups are homologically balanced but have no balanced presentation.
The examples constructed in \cite{HW85} include both finite groups and 
torsion-free groups.

\begin{lemma}
\label{hom bal}
A finitely presentable group $G$ is homologically balanced if and only if 
$\beta_2(G)=\beta_1(G)$ and $H_2(G;\mathbb{Z})$ is torsion-free.
\end{lemma}

\begin{proof}
We may assume that 
$H_1(G;\mathbb{Z})=G/G'\cong\mathbb{Z}^\beta\oplus{T}$ 
and 
$H_2(G;\mathbb{Z})\cong\mathbb{Z}^\beta\oplus{U}$, 
where $\beta=\beta_1(G)$ and $T$ and $U$ are finite.
If $A$ is an abelian group and $p$ is a prime let 
$r_p(A)=\mathrm{dim}_{\mathbb{F}_p}A/pA$.
Then $\beta_1(G;\mathbb{F}_p)=\beta+r_p(T)$
and $\beta_2(G;\mathbb{F}_p)=\beta+r_p(U)+r_p(Tor(T,\mathbb{Z}))$,
by the Universal Coefficient Theorem for homology.
Since $Tor(T,\mathbb{Z})\cong{T}$ (non-canonically),
$\beta_2(G;\mathbb{F}_p)=\beta_1(G;\mathbb{F}_p)$
for all primes $p$ if and only if $U=0$.
\end{proof}

Let $F(r)$ be the free group of rank $r$.
If $G=F(\beta)/\gamma_kF(\beta)$ then
$H_2(G;\mathbb{Z})\cong\gamma_kF(\beta)/\gamma_{k+1}F(\beta)$,
by the five-term exact sequence of low degree for the homology of
$G$ as a quotient of $F(\beta)$.
This abelian group has rank $\frac1k\Sigma_{d|k}\mu(d)\beta^{\frac{k}d}$,
where $\mu$ is the M\"obius function, 
by the Witt formulae \cite[Theorems 5.11 and 5.12]{MKS}.
Hence $H_2(G;\mathbb{Z})$ has rank $>\beta$ unless $\beta=1$ or $\beta=2$ 
and $k\leq3$ or $\beta=3$ and $k=2$.
Thus the only relatively free nilpotent groups with $\beta_2(G)\leq\beta_1(G)$ are  
$G\cong\mathbb{Z}^\beta$ with $\beta\leq3$ or 
$G\cong{F(2)/\gamma_3F(2)}$.
These groups have balanced presentations and Hirsch length $\leq3$.
Applying the Witt formulae inductively,
we see that $h(F(2)/\gamma_4F(2))=5$ and $h(F(2)/\gamma_5F(2))=8$.

A finitely generated nilpotent group has a maximal finite normal subgroup,
with torsion-free quotient \cite[5.2.7]{Rob}.
We may use the following lemma 
to extend arguments based on the torsion-free cases.

\begin{lemma}
\label{standard}
Let $G$ be a finitely generated nilpotent group,and let $T$ be
its maximal finite normal subgroup $T$. 
Then the natural epimorphism $p:G\to{G/T}$ induces isomorphisms
$H_i(p)=H_i(p;\mathbb{Q})$, for all $i\geq0$.
\end{lemma}

\begin{proof}
Since $T$ is finite, $H_i(T)=0$ for $i>0$.
Hence the LHS spectral sequence for the homology of $G$ as an extension of 
$G/T$ by $T$ collapses, and the edge homomorphisms $H_i(p)$ 
are isomorphisms for all $i\geq0$.
\end{proof}

We may also arrange that the abelianization is torsion-free.

\begin{lemma}
\label{free abelian}
Let $G$ be a finitely generated torsion-free nilpotent group,
and let $\beta=\beta_1(G)$.
Then $G$ has a subgroup $J$ of finite index which can be generated 
by $\beta$ elements and such that $H_i(J)\cong{H_i(G)}$, for all $i\geq0$.
\end{lemma}

\begin{proof}
Let $f:F(\beta)\to{G}$ be a homomorphism such that $H_1(f)$ 
is an isomorphism, and let $J=\mathrm{Im}(f)$.
Then $[G:J]$ is finite, and $J$ is subnormal in $G$ \cite[5.2.4]{Rob}.
The inclusion $j:J\to{G}$ induces a map between the LHS spectral sequences 
for $J$ and $G$ as central extensions of $J/J\cap\zeta{G}$ and $G/\zeta{G}$,
respectively.
Induction on the nilpotency class of $G$ now shows that 
$H_i(j)$ is an isomorphism for all $i\geq0$.
\end{proof}

We may also use induction on the nilpotency class to show that $J$ and $G$ 
have isomorphic Mal'cev completions.

\begin{lemma}
\cite{Lu83}
\label{min gen}
Let $G$ be a finitely generated nilpotent group.
Then $G$ can be generated by $d$ elements,
where $d=\max\{\beta_1(G;\mathbb{F}_p)|p~prime\}$.
\end{lemma}

\begin{proof}
If $G$ is abelian this is an easy consequence of the structure
theorem for finitely generated abelian groups.
In general, if $G$ is nilpotent and the image in $G/G'$ of a subset $X\subset{G}$ 
generates $G/G'$ then $X$ generates $G$ \cite[Lemma 5.9]{MKS}.
\end{proof}

The expression for $d$ is clearly best possible.

The following lemma was prompted by Lemma 2 of \cite{CJP}.

\begin{lemma}
\label{E^2_{1,1}}
Let $G$ be a group with a subgroup $Z\leq\zeta{G}\cap{G'}$ 
such that $z=h(Z)\geq1$,
and let $\overline{G}=G/Z$ and $R$ be $\mathbb{Z}$ or a field.
Let $\overline{\beta}_i=\beta_i(\overline{G};R)$, for $i\geq0$.
Then $\overline\beta_2\geq{z}$ and
\[
\overline{\beta}_2-z+
\max\{\overline{\beta}_1{z}-\overline{\beta}_3,0\}\leq\beta_2(G;R)\leq
\overline{\beta}_2-z+\overline\beta_1z+\binom{z}2.
\]
\end{lemma}

\begin{proof}
The quotient $\overline{G}$ acts trivially on the cohomology of $Z$,
since $Z$ is central in $G$.
Hence the rank of the $E^2_{1,1}$ term of the LHS spectral sequence  
\[
H_p(\overline{G};H_q(Z;R))\Rightarrow{H_{p+q}(G;R)}.
\]
for
the homology of $G$ as a central extension of $\overline{G}$ 
is $\overline{\beta}_1z$.
The $E^2_{2,0}$ term has rank $\overline{\beta}_2$,
the $E^2_{0,1}$ term has rank $z$,
the $E^2_{3,0}$ term has rank $\overline{\beta}_3$,
and the $E^2_{0,2}$ term has rank $\binom{z}2$.
The differential $d^2_{2,0}$ must be surjective, since $Z\leq{G'}$,
and so $\overline\beta_2\geq{z}$.
The lemma follows easily.
\end{proof}

When $Z\cong\mathbb{Z}$ the spectral sequence reduces to the Gysin sequence
associated to $G$ as an extension of $\overline{G}$ by $\mathbb{Z}$,
and the bounds given by Lemma \ref{E^2_{1,1}} simplify to
\[
\beta_2(\overline{G};R)-1\leq\beta_2(G;R)\leq
\beta_2(\overline{G};R)-1+\beta_1(\overline{G};R).
\]
Note also that if $h=h(\overline{G})\leq6$ (and $R=\mathbb{Q}$)
then $\beta_3(\overline{G})$ is determined by $\beta_1(\overline{G})$ 
and $\beta_2(\overline{G})$,
via Poincar\'e duality and the condition $\chi(\overline{G})=0$.

From another point of view, if 
$G\cong{F(\beta)/\gamma_kF(\beta)}$ 
then rearranging the right hand inequality gives the bound
\[
1-\beta+\frac1k\Sigma_{d|k}\mu(d)\beta^{\frac{k}d}\leq\beta_2(\overline{G}).
\]

\section{the case $\beta=3$}


Lubotzky has shown that if a nilpotent group $G$ can be generated by $d$ elements and $p$ is a prime such that $d=\beta_1(G;\mathbb{F}_p)$
(as in Lemma \ref{min gen}) then either
$\beta_2(G;\mathbb{F}_p)>\frac{d^2}4$ or
$G=1$, $\mathbb{Z}$ or $\mathbb{Z}^2$.
Hence homologically balanced 
nilpotent groups can be generated by 3 elements \cite{Lu83}.
There is a similar inequality for the rational Betti numbers.

\begin{lemma}
Let $G$ be a finitely generated nilpotent group such that $h(G)>2$.
Then $\beta_2(G)>\frac14\beta_1(G)^2$.
\end{lemma}

\begin{proof}
We may assume that $G$ can be generated by $\beta=\beta_1(G)$ elements, 
by Lemma \ref{free abelian}.
Then $\beta_1(G;\mathbb{F}_p)=\beta$, 
and so $\beta_2(G;\mathbb{F}_p)>\frac{\beta^2}4$, for all primes $p$
\cite{Lu83}.
Since $G$ is a finitely generated nilpotent group, 
$H_2(G;\mathbb{Z})$ is finitely generated,
and so the torsion subgroup of $H_2(G;\mathbb{Z})$ is finite.  
If $p$ does not divide the order of this torsion subgroup
then $\beta_2(G)=\beta_2(G;\mathbb{F}_p)$.
The lemma follows.
\end{proof}

An immediate consequence is that if $G$ is a finitely generated nilpotent group such that $\beta_2(G)\leq\beta_1(G)$ then $\beta_1(G)\leq3$.
We shall use the following result from commutative algebra 
to settle the case $\beta_1(G)=3$.

\begin{lem}
\cite[Lemma 1]{CEM}
Let $R$ be a Cohen-Macaulay local ring of Krull dimension $d$ 
and with residue field $k$.
Let $M$ be a nonzero $R$-module of finite length.
Then $\dim_k\mathrm{Tor}_1^R(k,M)\geq\dim_kk\otimes_RM+d-1$.
\qed
\end{lem}

We shall apply this result to modules over formal power series rings
$R=k[[x,y,z]]$, where $k$ is a prime field.
(Thus $k=\mathbb{Q}$ or $\mathbb{F}_p$ for some prime $p$.)
All such rings are regular, and hence are Cohen-Macaulay
\cite[page 187]{Ku}.

\begin{theorem}
\label{beta=3}
Let $G$ be a finitely generated nilpotent group such that 
$G/G'\cong\mathbb{Z}^3$.
If $G'\not=1$ then $\beta_2(G;k)>3$ for some prime field $k$.
\end{theorem}

\begin{proof}
Let $\{x,y,z\}$ be a fixed basis for $G/G'$.
If $G'\not=1$ then $G'/G''$ is non-zero, and is finitely generated.
Hence $A=H_1(G';k)\not=0$ for some prime field $k$.
Let $k\Lambda=k[G/G']$ and let $X=1-x$, $Y=1-y$ 
and $Z=1-z$.
Then $\mathcal{I}=(X,Y,Z)$ is the augmentation ideal in $k\Lambda$.
Since $G$ is nilpotent,  
the $k\Lambda$-module $A$ is annihilated 
by a power of $\mathcal{I}$, and $r=\dim_kA$ is finite.
Hence $A$ is a module of finite length over the $\mathcal{I}$-adic completion 
$R=\widehat{k\Lambda}$, 
which is a regular local ring of Krull dimension 3.
Moreover,  $A\cong{R}\otimes_{k\Lambda}A$.
Since completion is faithfully flat,
$H_i(G/G';A)=\mathrm{Tor}_i^{k\Lambda}(k,A)
\cong\mathrm{Tor}_i^R(k,A)$, for all $i$.
Let $b_i=\dim_k\mathrm{Tor}_i^R(k,A)$.
Then $b_0=\mathrm{dim}_kk\otimes_RA>0$, since $A\not=0$, and $b_1\geq{b_0+2}$,
by \cite[Lemma 1]{CEM}, since $R$ has Krull dimension $d=3$.

We may bound $\beta_2(G;k)$ below by means of the LHS spectral sequence 
for $G$ as an extension of $G/G'$ by $G'$.  
Since $\dim_\mathbb{Q}H_i(G/G';A)=b_i$, for all $i$,
the spectral sequence gives 
\[
\beta_2(G;k)\geq\max\{3-b_0,0\}+\max\{b_1-1,0\}
\geq\max\{3-b_0,0\}+b_0+1\geq4.
\]
This proves the Theorem.
\end{proof}

We note the following  corollaries.

\begin{cor}
Let $G$ be a finitely generated nilpotent group with torsion subgroup $T$
and such that $\beta_1(G)=3$.
Then the following are equivalent:
\begin{enumerate}
\item $G/T\cong\mathbb{Z}^3$;
\item$h(G)=3$;
\item$\beta_2(G)=\beta_1(G)=3$.
\end{enumerate}
\end{cor}

\begin{proof}
Clearly (1) and (2) are equivalent and imply (3),
while if (3) holds then $G'$ is finite,  and so (1) holds, by the Theorem.
\end{proof}

\begin{cor}
If $G$ is a  nilpotent group with a balanced presentation and such that  
$\beta_1(G)=3$ then $G\cong\mathbb{Z}^3$.
\qed
\end{cor}

Freedman, Hain and Teichner have also shown that
$\beta_2(G)>\frac14\beta_1(G)^2$ (for nilpotent groups $G$ such that $h(G)>2$),
and have used the Dwyer filtration \cite{Dw75}
to refine this estimate for certain groups \cite{FHT}.
The following assertion is a paraphrase of one of their main results. 

\begin{thm}
 \cite[Theorem 2]{FHT}
If $G$ is a finitely generated nilpotent group such that $h(G)>2$ and 
$G/\gamma_rG\cong{F(\beta)/\gamma_rF(\beta)}$ for some $r\geq2$
then $\beta_2(G)>\frac{(r-1)^{r-1}}{r^r}\beta_1(G)^r$.
\qed
\end{thm}

Setting $r=2$ recovers the inequality $\beta_2(G)>\frac14\beta_1(G)^2$,
since we may assume that $G/G'\cong\mathbb{Z}^\beta$,
by Lemma \ref{free abelian}.

In seeking examples of nilpotent groups $G$ with $\beta_2(G)=\beta_1(G)=2$,
we may focus on groups that can be generated by 2 elements,
by Lemma \ref{free abelian}.
If $G$ is such a group then $G/\gamma_2G=F(2)/\gamma_2F(2)$,
and $G/\gamma_rG$ is a quotient of $F(2)/\gamma_rF(2)$,
for all $r>2$.
If $G/\gamma_3G$ is a proper quotient of $F(2)/\gamma_3F(2)$
then $\gamma_2G/\gamma_3G$ is finite,  
since $\gamma_2(F(2)/\gamma_3F(2)\cong\mathbb{Z}$.
Hence $G'=\gamma_2G$ is finite and so $\beta_2(G)=1$.
Thus we may assume that $G/\gamma_3G\cong{F(2)/\gamma_3F(2)}$.
On the other hand,  if $G/\gamma_5G\cong{F(2)/\gamma_5F(2)}$ 
then setting $\beta=2$ and $r=5$ in the inequality of
\cite[Theorem 2]{FHT} gives $\beta_2(G)>\beta_1(G)$.
Thus we may assume that $\gamma_4G/\gamma_5G$ has rank $\leq2$.
(The argument for \cite[Theorem 2]{FHT} may be tweaked to show that 
$\gamma_4G/\gamma_5G$ must have rank 1.)

The authors of \cite{FHT} conjecture that there is no nilpotent
group $G$ such that $G/\gamma_4G\cong{F(2)/\gamma_4F(2)}$ and
$\beta_2(G)=\beta_1(G)$.

\section{hirsch length $\leq5$}

As our primary interest is in the torsion-free case, 
and the cases with $\beta\geq3$ or $h\leq3$ have been settled,
We shall assume henceforth that $\beta=2$ and $h\geq4$.
The next theorem was the origin of this paper.

\begin{theorem}
\label{h=4}
There is just one torsion-free nilpotent group of Hirsch length $4$
which is homologically balanced.
\end{theorem}

\begin{proof}
Let $N$ be a torsion-free nilpotent group of Hirsch length $4$.
If  $F$ is a field then $\beta_1(N;F)\geq2$, 
since $N$ is not virtually cyclic, and $\beta_2(N;F)=2(\beta_1(N;F)-1)$, 
since $N$ is an orientable $PD_4$-group and $\chi(N)=0$.
Hence if $\beta_2(N;F)\leq\beta_1(N;F)$ then $\beta_1(N;F)=2$.
If $N$ is homologically balanced this holds for all fields $F$,
and so $N/N'\cong\mathbb{Z}^2$.
Therefore $N$ can be generated by 2 elements.

Since $N$ is a quotient of $F(2)$,
the quotient $\gamma_2N/\gamma_3N$ is cyclic.
If $\gamma_2N/\gamma_3N$ were finite then all subsequent 
factors of the lower central series would be finite \cite[5.2.5]{Rob},
and so $h(N)=2$, contradicting the hypothesis $h(N)=4$. 
Therefore $\gamma_2N/\gamma_3N\cong\mathbb{Z}$ 
and so $N/\gamma_3N\cong{F(2)/\gamma_3F(2)}$.
Hence $\gamma_3N\cong\mathbb{Z}$.
Since $N$ is nilpotent, it follows that $\gamma_3N$ is central in $N$.
Hence $N$ has a presentation
\[
\langle{x,y,u,z}\mid{u=[x,y]},~[x,u]=z^a,~[y,u]=z^b,~xz=zx,~yz=zy\rangle,
\]
in which $x,y$ represent a basis for $N/\gamma_3N$ and $z$ represents
a generator for $\gamma_3N$.
Since $N/N'\cong\mathbb{Z}^2$ we must have $(a,b)=1$,
and after a change of basis for $F(2)$ we may assume that $a=1$ and $b=0$.
The relation $yz=zy$ is then a consequence of the others, 
and so the presentation simplifies to
$\langle{x,y}\mid [x,[x,[x,y]]]=[y,[x,y]]=1\rangle.$
\end{proof} 

An alternative argument uses the observations that $N'\cong\mathbb{Z}^2$ 
(since it is torsion-free nilpotent and $h(N')=2$) and that the image of $N$ 
in $\mathrm{Aut}(N')$ preserves the flag $1<\gamma_3N<N'$.
Since $N$ is nilpotent and $N/N'\cong\mathbb{Z}^2$ this image is infinite cyclic, 
and so $N\cong{C}\rtimes\mathbb{Z}$, 
where $C$ is the centralizer of $N'$ in $N$.
Moreover,  $C\cong\mathbb{Z}^3$, since $C/N'\cong\mathbb{Z}$.
Hence $N\cong\mathbb{Z}^3\rtimes_A\mathbb{Z}$,
where $A\in{SL(3,\mathbb{Z})}$ is a triangular matrix.
It is not hard to see that since $N/N'$ is torsion-free of rank 2
there is a basis for $C$ such that
\[
A=
\left(
\begin{matrix}
1& 0 & 0\\
1 & 1 & 0\\
0& 1 & 1
\end{matrix}
\right)\in{SL(3,\mathbb{Z})}.
\]
Hence $N$ has the 2-generator balanced presentation
\[
\langle{t,u}\mid[t,[t,[t,u]]]=[u,[t,u]]=1\rangle,
\]
in which $u$ corresponds to the column vector $(1,0,0)^{tr}$ in $\mathbb{Z}^3$.

\begin{cor}
\label{2 gen cor}
The group $N$ described in Theorem \ref{h=4} is the only torsion-free nilpotent group of Hirsch length $4$ which can be generated by $2$ elements.
\end{cor}

\begin{proof}
If $G$ is a  nilpotent group which can be generated by 2 elements 
then either $G/G'\cong\mathbb{Z}^2$ or $h(G)\leq1$.
Hence if $h(G)=4$ then $\beta_1(F)=2$ for all fields $F$.
Since $G$ is an orientable $PD_4$-group and $\chi(G)=0$ it follows that 
$G$ is homologically balanced, and so $G\cong{N}$.
\end{proof}

This group is a quotient of 
$\widetilde{F}=F(2)/\gamma_4F(2)$ by 
a maximal infinite cyclic subgroup of $\zeta\widetilde{F}$.
The quotient of $\widetilde{F}$ by the subgroup generated by the $p$th power 
of a non-trivial central element is  nilpotent of Hirsch length 4, 
and has a finite normal subgroup of order $p$.
Thus the condition on torsion is necessary for this corollary.
We do not know whether every homologically balanced nilpotent group 
of Hirsch length 4 is torsion-free.

We remark here that there is just one indecomposable nilpotent 
Lie algebra of dimension 4.
This derives from the Lie group $G_4$ of \cite{Ni83}.
(In the context of Thurston geometries, this is also known as $Nil^4$ 
\cite[Chapter 7]{Hi}.)

\begin{theorem}
\label{h5b2}
Let $G$ be a finitely generated nilpotent group 
such that $\beta_1(G)=2$ and $h(G)=5$.
Then $\beta_2(G)=3$.
\end{theorem}

\begin{proof}
We may assume that $G$ is torsion-free, by Lemma \ref{standard}.

The intersection $G'\cap\zeta{G}$ is non-trivial \cite[5.2.1]{Rob},
and so we may choose a maximal infinite cyclic subgroup $A\leq{G'\cap\zeta{G}}$.
Let $\overline{G}=G/A$.
Then $h(\overline{G})=4$ and $\overline{G}$ is also torsion-free, 
since the preimage of any finite subgroup in $G$ is torsion-free 
and virtually $\mathbb{Z}$.
Hence $\overline{G}$ is an orientable $PD_4$-group.
Moreover, $\beta_1(\overline{G})=2$, since $A\leq{G'}$, and
so $\overline{G}^\tau/[\overline{G},\overline{G}^\tau]\not=1$.
 
The group $G$ is a central extension of $\overline{G}$ by $\mathbb{Z}$.
The LHS cohomology spectral sequence associated to this extension 
reduces to a ``Gysin" exact sequence 
\cite[Example 5C]{McC}:
\[
0\to\mathbb{Q}e\to{H^2(\overline{G})}\to{H^2(G)}\to
{H^1(\overline{G})}\xrightarrow{\cup{e}}{H^3(\overline{G})}\to\dots,
\]
where $e\in{H^2(\overline{G};\mathbb{Z})}$ classifies the extension.
If $J$ is any finitely generated group then the kernel of the homomorphism 
$\psi_J:\wedge^2H^1(J)\to{H^2(J)}$ induced by cup product is isomorphic to 
$\mathrm{Hom}(J^\tau/[J,J^\tau],\mathbb{Q})$ \cite{Hi87}.
In our case $\wedge^2H^1(G)\to{H^2(G)}$ has rank 1, 
since $\beta_1(\overline{G})=2$.
Hence $\psi_{\overline{G}}=0$, since
$\overline{G}^\tau/[\overline{G},\overline{G}^\tau]\not=1$.
Since the cup product of $H^3(\overline{G})$ with
$H^1(\overline{G})$ is a non-singular pairing,
it follows that $\alpha\cup{e}=0$ for all
$\alpha\in{H^1(\overline{G})}$.
Hence $\beta_2(G)=2-1+2=3$. 
\end{proof}

We shall use the following extension of this result at the end of \S5.

\begin{cor}
\label{h5b2modp}
Let $G$ be a finitely generated, torsion-free nilpotent group 
such that $G/G'\cong\mathbb{Z}^2$ and $h(G)=5$.
Then $\beta_2(G;\mathbb{F}_p)=3$ for all primes $p$.
\end{cor}

\begin{proof}
The argument is essentially as in the theorem, 
but the appeal to \cite{Hi87} deserves comment.
We replace $J^\tau$ by the subgroup of $J$ generated by 
$J'$ and all $p$th powers;
if $p=2$ we note also that if $z\in{H^1(\overline{G};\mathbb{F}_2)}$
then $z$ is the reduction of an integral class, and so $z^2=0$.
\end{proof}

Theorem \ref{h5b2} implies that there are no homologically balanced groups
nilpotent groups of Hirsch length 5.

\begin{cor}
\label{h=5}
Let $N$ be a finitely generated nilpotent group of Hirsch length $h(N)=5$.
Then $\beta_2(N)>\beta_1(N)$.
\end{cor}

\begin{proof}
Since we may assume that $\beta_1(N)\leq3$,
this follows immediately from Theorems \ref{beta=3} and \ref{h5b2}.
\end{proof}

\section{metabelian nilpotent groups}

We shall settle the case of metabelian nilpotent groups
with the next theorem.

\begin{theorem}
\label{metabelian2}
Let $G$  be a finitely generated metabelian nilpotent group. 
If $h(G)>4$ then $\beta_2(G)>\beta_1(G)$.
\end{theorem}

\begin{proof}
We may assume that $G$ is torsion-free and can be generated by 2 elements, 
by Lemmas \ref{standard} and \ref{free abelian}, and Theorem \ref{beta=3}.

Hence $G/G'\cong\mathbb{Z}^2$ and $G'\cong\mathbb{Z}^r$,
with $r=h(G)-2>2$.
We shall again bound $\beta_2(G)$ below by means of the LHS spectral sequence
for $G$ as an extension of $G/G'$.  
Let $A=H_1(G')=\mathbb{Q}\otimes{G'}\cong\mathbb{Q}^r$,
and let $b_i=\dim_\mathbb{Q}H_i(G/G';A)$, for $i\geq0$.
The homology groups $H_i(G/G';A)$ may be computed from the complex
\[
0\to{A}\xrightarrow{\partial_2}{A^2}\xrightarrow{\partial_1}{A}\to0
\]
arising from the standard resolution of the augmentation
$\mathbb{Q}[G/G']$-module.
Since this complex has Euler characteristic 0, $b_0-b_1+b_2=0$.

Since $G$ is nilpotent and $G'$ is infinite,  $h(G/\gamma_3G)=3$,
and so $b_0=1$.
Hence $b_1=b_2+1$.
Let $\rho=\dim_\mathbb{Q}H_0(G/G';\wedge^2A)$.
Since $G'$ is abelian, $H_2(G')\cong\wedge^2A$,
and so has rank $\binom{r}2$.
It is not hard to see that $\rho\geq2$ if $r>3$.
There are no nonzero differentials in the spectral sequence
which begin or end at the $(1,1)$ position, and so 
\[
\beta_2(G)\geq{b_1+\max\{\rho-b_2,0\}}\geq\rho+1>2.
\]
If $r=3$ then then $\dim_\mathbb{Q}H_0(G/G';\wedge^2A)=1$, 
but the conclusion still holds, by Theorem \ref{h=5}.
(In this case the differential $d^2_{2,1}$ must be 0.)
\end{proof}

The theorem is not true  if $h(G)\leq4$, as is shown by the groups $\mathbb{Z}^2$,
$F(2)/\gamma_3F(2)$ and the group of Theorem \ref{h=4}.

If $G$ is any group then $G''=[G',G']\leq\gamma_4G$,
and so all nilpotent groups of class 3 are metabelian.
This can be improved slightly for 2-generator groups,
since $F(2)/\gamma_5F(2)$ is metabelian.

\begin{cor}
\label{class 5;2}
Let $G$  be a nilpotent group of class $\leq4$ and such that $\beta_1(G)=2$.
If $h(G)>4$ then  $\beta_2(G)>\beta_1(G)$.
\end{cor}

\begin{proof}
We may again assume that $G$ can be generated by $2$ elements, 
by Lemma \ref{free abelian}.
As it is then a quotient of $F(2)/\gamma_5F(2)$, 
it is metabelian,
and so the theorem applies.
\end{proof}

\section{hirsch length 6}


In this section we shall give an example of a torsion-free nilpotent group 
of Hirsch length 6 which is homologically balanced.
The construction was inspired by some of the features of the examples 
of nilpotent Lie algebras reported to us by Grant Cairns,
but stays within the realm of group theory.
The example is a central extension of a group of Hirsch length 5 by 
$\mathbb{Z}$,
and the claim is verified by means of the Gysin sequences
(with coefficients in a field) associated to the extension.

Let $K$ be the central extension of $A=\mathbb{Z}^4$ by $\mathbb{Z}$,
with presentation
\[
\langle{y,c,d,e,f}\mid[y,d]=[y,e]=f, ~[c,d]=f^{-1},~[y,c]=[c,e]=[d,e]=1,,
\]
\[
f~central
\rangle.
\]
We may define an automorphism $\theta$ of $K$ by setting
$\theta(y)=cy$, $\theta(c)=dc$, $\theta(d)=ed$, $\theta(e)=e$ and $\theta(f)=f$.
(It is straightforward to check that this definition is compatible with the relations
for $K$.)
The semidirect product $G=K\rtimes_\theta\mathbb{Z}$ has the presentation
\[
\langle{x,y,c,d,e,f}\mid{c=[x,y]},~d=[x,c],~e=[x,d],~[y,d]=[y,e]=f,
\]
\[
[c,d]=f^{-1},~[x,e]=[y,c]=[c,e]=[d,e]=[x,f]=[y,f]=1\rangle.
\]
It is clearly nilpotent, and torsion-free of Hirsch length 6,
and has abelianization $G/G'\cong\mathbb{Z}^2$ and centre 
$\zeta{G}=\langle{f}\rangle\cong\mathbb{Z}$.
Note also that $\gamma_4G\cong\mathbb{Z}^2$ and so $G/\gamma_4G$
is a proper quotient of $F(2)/\gamma_4F(2)$.

Let $\bar\theta$ be the automorphism of $A$ induced by $\theta$,
and let $\overline{G}=G/\zeta{G}\cong{A}\rtimes_{\bar\theta}\mathbb{Z}$.
Then $\overline{G}$ is a torsion-free nilpotent group of Hirsch length 5,
and so $\beta_2(\overline{G})=3$, by Theorem \ref{h5b2}.
The group $G$ is a central extension of $\overline{G}$ by $\mathbb{Z}$.
Let $\eta\in{H^2(\overline{G};\mathbb{Z})}$ classify this extension.

Let $\{x^*,y^*\}$ be a basis for 
$H^1(\overline{G};\mathbb{Z})=\mathrm{Hom}(\overline{G},\mathbb{Z})$ 
which is Kronecker dual to the basis for $G/G'$ represented by $\{x,y\}$.
Then $x^*|_A=0$, while $y^*|_A$ is trivial on $\{c,d,e\}$.
Let $\{y^*|_A, c^*, d^*, e^*\}$ be the basis for 
$H^1(A;\mathbb{Z})=\mathrm{Hom}(A,\mathbb{Z})$
which is Kronecker dual to $\{y,c,d,e\}$.
The matrix of $M^*=H^1(M)$ with respect to the basis
$\{y^*, c^*, d^*, e^*\}$ is the transpose of the matrix for $M$ 
with respect to $\{y,c,d,e\}$,
and so $M^*(y^*)=y^*$, $M^*(c^*)=c^*+y^*$, $M^*(d^*)=d^*+c^*$
and $M^*(e^*)=e^*+d^*$.
We shall henceforth write $y^*$ instead of $y^*|_A$, where appropriate,
and  drop the $\cup$ sign for cup products.
Then $H^2(A;\mathbb{Z})$ has basis 
$\{y^*c^*,y^*d^*,y^*e^*,c^*d^*,c^*e^*,d^*e^*\}$.

We may identify the restriction $\eta|_A$ by considering the induced extensions
of  2-generator subgroups of $A$, 
and we see that $\eta|A= y^*d^*+y^*e^*-c^*d^*$.
Then $\eta^2\not=0$, since 
$\eta^2|_A=-2y^*c^*d^*e^*$ generates $H^4(A)$.
Similarly, $y^*\eta|_A=-y^*c^*d^*\not=0$.

We claim that $x^*\eta^2\not=0$.
(This is most easily seen topologically. 
The group $G$ is the fundamental group of the mapping torus of 
a homeomorphism of the 4-torus, with fundamental group $A$.
The 4-torus intersects a loop representing $x$ transversely in one point,
and so represents the Poincar\'e dual to $x^*$.
Thus $x^*\eta^2\cap[G]=\eta^2\cap(x^*\cap[G])=\eta^2|_A\cap[A]=\pm1$.
There is a more algebraic argument based on \cite{Ba80}.) 

Since $\eta^2|_A=2y^*\eta|_Ae^*$ and $x^*\eta^2\not=0$, 
it follows that $x^*y^*\eta\not=0$.
On the other hand, $x^{*2}\eta=y^{*2}\eta=0$,
and so $x^*\eta$ and $y^*\eta$ are linearly independent.
Hence cup product with $\eta$ is injective on $H^1(\overline{G})$.
It now follows from the exactness of the Gysin sequence that
$\beta_2(G)=2$.

The same argument with coefficients $F=\mathbb{F}_p$ shows that
$\beta_2(G;\mathbb{F}_p)=2=\beta_1(G;\mathbb{F}_p)$ for all primes $p$,
and so $G$ is homologically balanced.
(It is at this point that we need Corollary \ref{h5b2modp}.)

It can be shown that the relations $[c,e]=[d,e]=1$, $[c,d]=f^{-1}$ and $[y,d]=f$
in the above presentation for $G$ are consequences of the other relations.
Using four of the remaining relations to eliminate the generators $c,d,e,f$
leads to a presentation with 2 generators and 4 relations.
However, we do not know whether $G$ has a balanced presentation.

\section{nilpotent lie algebras and lie groups}


In the parallel world of nilpotent real Lie algebras $\mathfrak{n}$ there 
is a similar issue, of whether there are such algebras $\mathfrak{n}$ with 
$\beta_2(\mathfrak{n})=\beta_1(\mathfrak{n})$ and of large dimension.
There are just 3 non-abelian nilpotent Lie algebras $\mathfrak{n}$
of dimension $\leq7$ with $\beta_2(\mathfrak{n})=\beta_1(\mathfrak{n})$,
and for these the common value is 2 \cite{CJP}.
One has dimension 4 and the other two have dimension 6.
The two 6-dimensional nilpotent Lie algebras with $\beta_2=\beta_1$ 
each have 1-dimensional centre,
and they may be distinguished by the dimensions of the maximal 
abelian Lie ideals of the quotients by the centre, which are 4 and 3, respectively.
These Lie algebras correspond to the Lie groups $G_4$,
$G_{6,12}$ and $G_{6,14}$ of \cite{Ni83}.

There are direct connections between these aspects of nilpotency.
A finitely generated torsion-free nilpotent group $G$ has a
Mal'cev completion $G_\mathbb{R}$, which is a 1-connected
nilpotent real Lie group of dimension $h(G)$ in which $G$ is a lattice
(i.e., a discrete subgroup such that the coset space $G_\mathbb{R}/G$
is compact). 
Conversely,
every nilpotent real Lie algebra of dimension $\leq6$ is isomorphic to one with rational structure coefficients, and so the corresponding 1-connected Lie group admits lattices, by Mal'cev's criterion \cite[Chapter 2]{Rag}.
The coset space $M=G_\mathbb{R}/G$ is aspherical, 
and so $H^q(G;\mathbb{R})\cong{H^q_{DR}(M)}$,
the de Rham cohomology of $M$.
This is in turn the cohomology of the associated Lie algebra
$\mathfrak{g}$ \cite{No54}.

We may find further examples of torsion-free nilpotent groups $G$ 
such that $h(G)=6$ and $\beta_2(G)=\beta_1(G)=2$ among lattices
 in $G_{6,12}$ and $G_{6,14}$.
Each such lattice $\Gamma$ is nilpotent and of Hirsch length 6,
and we may invoke \cite{No54} to conclude that 
$\beta_2(\Gamma)=\beta_1(\Gamma)=2$.
(However we do not know whether such lattices must be homologically
balanced.
We remark that $G$ is the group constructed in \S5 then $G_\mathbb{R}\cong{G_{6,12}}$,
since $\overline{G}=G/\zeta{G}$ has an abelian normal subgroup 
of rank 4.)

Some information is lost in passing from groups to Lie algebras,
since commensurable nilpotent groups have isomorphic completions.
Moreover,  
from dimension 7 onwards there are uncountably many isomorphism classes 
of nilpotent Lie algebras,
and so most do not derive from discrete groups.
The main result of \cite{CJP} is that there are no nilpotent Lie algebras 
$\mathfrak{n}$ of dimension 7 with 
$\beta_2(\mathfrak{n})=\beta_2(\mathfrak{n})=2$,
and so there are no such groups of Hirsch length 7.

It remains an open question whether there are any examples at all of 
homologically balanced nilpotent groups of Hirsch length $>7$.

\medskip
\noindent{\bf Acknowledgment.}
I would like to thank Grant Cairns for providing me with a copy of \cite{CJP},
and for subsequent advice.



\begin{thebibliography}{99}

\bibitem{Ba80} Barge, J. Dualit\'e dans les rev\^etements galoisiens,

Invent. Math. 58 (1980), 101--106.

\bibitem{Bi} Bieri, R. \textit{Homological Dimensions of Discrete Groups},

QMC Lecture Notes, London (1976).

\bibitem{CJP} Cairns, G., Jessup, B. and Pitkethly, J. On the Betti
numbers of nilpotent Lie algebras of small dimension,
in \textit{Integrable Systems and Foliations},

Progress in Mathematics 145, Birkh\"auser Verlag (1997), 19--31.

\bibitem{CEM} Charalambous, H., Evans, E. G. and Miller, C.
Betti numbers for modules of finite length,
Proc. Amer. Math. Soc. 109 (1990), 63--70.

\bibitem{Dw75} Dwyer, W. G. Homology, Massey products and maps between groups,

J. Pure Appl. Alg. 6 (1975), 177--190.

\bibitem{FHT} Freedman, M., Hain, R. and Teichner, P.  Betti number estimates for nilpotent groups,
in \textit{Fields Medallists Lectures} (edited by M. Atiyah and D. Iagolnitzer),

WSP Series in 20th Century Mathematics 5,

World Scientific Publications, River Edge NJ -- Singapore (1997), 413--434.

\bibitem{HW85} Hausmann, J.-C. and Weinberger, S.  Caract\'eristiques d'Euler et
groupes 

fondamentaux des vari\'et\'es en dimension 4,

Comment. Math. Helv.  60 (1985), 139--144.

\bibitem{Hi} Hillman, J.  A. \textit{Four-Manifolds, Geometries and Knots},

Geometry and Topology Monographs, vol. 5,

Geometry and Topology Publications (2002). (Revisions 2007 and 2014).

 \bibitem{Hi87} Hillman, J.  A.  The kernel of integral cup product,

J. Austral. Math. Soc. 43 (1987), 10--15.

\bibitem{Hi19} Hillman, J.  A.  $3$-manifolds with nilpotent embeddings in $S^4$,

J.Knot Theory Ramif. 29 (2020),  2050094 (7 pages).

\bibitem{Ku} Kunz, E. \textit{Introduction to Commutative Algebra},

Birkh\"auser, Boston -- Basel -- Stuttgart (1985).

\bibitem{Lu83} Lubotzky, A. Group presentation, $p$-adic analytic groups 
and lattices in $SL_2(\mathbb{C})$,
Ann. Math. 118 (1983), 115--130.

\bibitem {McC} McCleary, J. 
\textit{User's Guide to Spectral Sequences},

Mathematics Lecture Series 12, 
Publish or Perish, Inc., Wilmington (1985).

\bibitem{MKS} Magnus, W., Karras, A. and Solitar, D.  \textit{Combinatorial Group
Theory},

Interscience Publishers, New York - London - Sydney (1966).

Second revised edition, Dover Publications inc, New York (1976).

\bibitem{Ni83} Nielsen, O. A. Unitary representations and coadjoint orbits of
nilpotent Lie groups,
Queen's Papers in Pure  and Applied Mathematics 63,

Queen's University, Kingston, Ontario (1983). (117 pages).

\bibitem{No54} Nomizu, K. On the cohomology of compact homogeneous spaces
of nilpotent Lie groups,
Ann. Math. 59 (1954), 531--538.

\bibitem{Rag} Raghunathan, M. S. \textit{Discrete Subgroups of Lie Groups},

Ergeb. Math.  Bd 68, Springer-Verlag, Berlin-- Heidelberg --New York (1972).

\bibitem {Rob} Robinson, D. J. S. 
\textit{A Course in the Theory of Groups},

Graduate Texts in Mathematics 80, 

Springer-Verlag, Berlin - Heidelberg - New York (1982).    

\end{thebibliography}
\end{document}